\documentclass{article}
\usepackage{amsfonts}\usepackage{amssymb}
\usepackage{graphicx}\usepackage{amsmath}
\usepackage{abstract}%\usepackage{keywords}
\usepackage{color}
\usepackage{graphics}   \usepackage{amsthm}%for: \begin{proof}
\newtheorem{theorem}{Theorem}[section]
\newtheorem{lemma}[theorem]{Lemma}

\newtheorem{corollary}[theorem]{Corollary}
\numberwithin{equation}{section}
\providecommand{\keywords}[1]{\textbf{Key words.} #1}

\title{A mixed-element two-grid discretization for Helmholtz transmission eigenvalues}

\author{Yao Luo, Yidu Yang%\thanks{\tt School of Mathematics and Computer Science, Guizhou Normal University,
 \\\\
{\small School of Mathematics and Computer Science, }\\{\small
Guizhou Normal University,  Guiyang,  $550001$,  China}\\{\small
447004837@qq.com,
 ydyang@gznu.edu.cn}
}
\begin{document}
\date{}
\maketitle
\begin{abstract}The Helmholtz transmission eigenvalue problem has received much concern in materials science,
so it's significant to explore the efficient calculational method of
the problem to mathematics and mechanics community. In this paper,
based on a variational formulation proposed by Cakon, Monk and Sun, we introduce a
mixed-element two-grid discretization and prove error estimates for
this method theoretically.
 Some numerical results are presented to confirm the theoretical analysis and  show that the method here is efficient.
\end{abstract}
%\begin{keywords}

\keywords{transmission eigenvalues, finite element method, a
mixed-element two-grid discretization, error estimates}
%\end{keywords}

%\pagestyle{myheadings} \thispagestyle{plain} \markboth{YIDU YANG,
%JIAYU HAN AND HAI BI }{ FINITE ELEMENT METHOD FOR TRANSMISSION
%EIGENVALUES}

\section{Introduction}
\indent The transmission eigenvalue problem is a hot topic in mathematics and mechanics community because
it has widely applications in materials science. For instance, the measured transmission eigenvalues
 can be used to estimate properties of the scatterer \cite{1,2}, and the transport through a quantum-scale
  device may be uniquely characterized by its transmission eigenvalues \cite{3}.
   For that reason, there exist many researches such as \cite{colton2,ji,monk,8} for the numerical treatments of transmission eigenvalue problem.\\
\indent We know that the transmission eigenvalue problem is difficult to solve because it is a quadratic and non-selfadjoint problem
and its knowledge system is not covered by the standard theory of eigenvalue problems \cite{8}. To solve the problem,
\cite{9} uses multigrid method with one correction step on each iteration, \cite{sun1} applies iterative method to compute the transmission eigenvalues. Both of them transform the problem
 to a series of selfadjoint eigenvalue problems. However, the two-grid discretization could be used to solve
  non-slefadjoint finite eigenvalue problems directly. It is first introduced by Xu \cite{10}, then popularized on many eigenvalue problems
  (see \cite{11,12,19,zhou,yang2,yang3,13}). Recently \cite{13} utilizes a $H^{2}$ conforming element two-grid discreization
  to compute the transmission eigenvalues successfully.
  From it we know before using the two-grid discretization we need a good variational formulation firstly.
  At this field, Cakoni, Monk and Sun \cite{8} bring forth a new weak formulation and make the error analysis for transmission eigenvalues of the finite element approximation. Based on the framework of \cite{8},
  we propose a mixed-element two-grid discretization to calculate these eigenvalues.\\
\indent Now we introduce the characteristics of our method. The main idea is to transform a eigenvalue problem to two eigenvalue problems
on a coarse grid $\pi_H$ and two boundary value problems on a fine grid $\pi_h$. And the coefficient matrices of the two problems on the grid $\pi_h$
are the same such that we don't need much time to assemble the matrices respectively. In addition, this method can
keep fast rate of convergence which will be proved by our numerical results later. To analyse the correctness of mixed-element
 two-grid discretization we shall do the error estimates for eigenfunctions in norm $\|\cdot\|_{H_{0}^{1}(\Omega)\times L^{2}(\Omega)}$.
  Inspired by \cite{13}, we prove it by using Aubin-Nitsche technique and we also discuss the error estimate of eigenvalues under the condition of $n\in L^\infty(\Omega)$.
  By the way, the estimates in norm $\|\cdot\|_{H_{0}^{1}(\Omega)\times L^{2}(\Omega)}$ is very important
  to a posteriori error estimates which is the foundation of adaptive method. \\
\indent In this paper, we need the basic theory of finite element
methods of \cite{14,15,16}. \\
\indent Suppose that $C$ is a positive
constant independent of $h$, which may not be the same constant in
different places.
For simplicity, we use symbol $a\lesssim b$ to replace $a\leq Cb$.\\

\section{The finite element method}
\indent The Helmholtz transmission eigenvalue problem
is: Find $k\in \mathbb{C}$, $w, \sigma\in L^{2}(\Omega)$,
$w-\sigma\in H^{2}(\Omega)$ which satisfy
\begin{eqnarray}\label{ss2.1}
&&\Delta w+k^{2}n(x)w=0,~~~in~ \Omega,\\\label{ss2.2}
 &&\Delta
\sigma+k^{2}\sigma=0,~~~in~ \Omega,\\\label{ss2.3}
 &&w-\sigma=0,~~~on~ \partial
\Omega,\\\label{ss2.4}
 &&\frac{\partial w}{\partial \nu}-\frac{\partial
\sigma}{\partial \nu}=0,~~~ on~\partial \Omega ,
\end{eqnarray}
where $\Omega \subset \mathbb{R}^2$ is a bounded Lipschitz domain,
real valued function $n \in L^\infty(\Omega)$ such that $n-1$ is
strictly postive (or strictly negative) almost everywhere in
$\Omega$.\\

 \indent Regarding the problem (\ref{ss2.1})-(\ref{ss2.4})
papers \cite{cakoni3,7} transform it to a quadratic eigenvalue
problem: Find $k^{2}\in \mathbb{C}$, $k^{2}\not=0$, $u=w-\sigma\in
H_{0}^{2}(\Omega)$ such that
\begin{eqnarray}\label{ss2.5}
\int\limits_{\Omega}\frac{1}{n(x)-1}(\Delta u+k^{2} u)(\Delta \overline{v}+k^{2} n(x)\overline{v})dx=0,~~~\forall v \in~H_{0}^{2}(\Omega).
\end{eqnarray}
\indent Assume that there exists some constant $\delta >0$  such
that
 \begin{eqnarray}
 1+\delta\leq \inf\limits_{\Omega}n(x)\leq n(x)\leq \sup\limits_{\Omega}n(x)<\infty.\nonumber
 \end{eqnarray}
Then we employ some symbols from \cite{8}, let $$(u,v)_{n-1}=\int\limits_{\Omega}\frac{1}{n-1}u\overline{v}dx,$$
$$(u,v)=\int\limits_{\Omega}u\overline{v}dx.$$
And we define the sesquilinear form $A$ on $(H_{0}^{2}(\Omega)\times H_{0}^{1}(\Omega))\times (H_{0}^{2}(\Omega)\times H_{0}^{1}\Omega)$ by
\begin{eqnarray}
\label{s2.1}&&A((u,w),(v,z))=(\Delta u,\Delta v)_{n-1}+(\nabla w,\nabla z).
\end{eqnarray}
\indent For convenience, define Hilbert space $\mathbf{H}=H_{0}^{2}(\Omega)\times H_{0}^{1}(\Omega)$ with norm $\|(u,w)\|_{\mathbf{H}}=\|u\|_{2}+\|w\|_{1}$($\|\cdot\|_{l}$ is the norm of Sobolev space $H^{l}(\Omega)$), and define $\mathbf{H_{1}}=H_{0}^{1}(\Omega)\times L^{2}(\Omega)$ with norm $\|(u,w)\|_{\mathbf{H_{1}}}=\|u\|_{1}+\|w\|_{0}$.\\
\indent Note that $A(\cdot,\cdot)$ is an inner product on $\mathbf{H}=H_{0}^{2}(\Omega)\times H_{0}^{1}(\Omega)$,
norm $\|\cdot\|_{A}=A(\cdot,\cdot)^{\frac{1}{2}}$ is equivalent to the norm $\|\cdot\|_{\mathbf{H}}$,
 and
$\mathbf{H}\hookrightarrow \mathbf{H_{1}}$
compactly.\\
 \indent Let
 \begin{eqnarray}
 B((u,w),(v,z))=-((u,\Delta
v)_{n-1}+(\Delta u,nv)_{n-1}-(\nabla w,\nabla v)+(nu,z)_{n-1}).\label{s2.24}
\end{eqnarray}
When $n\in L^{\infty}(\Omega)$,
\begin{eqnarray}
&&~~~|B((u,w),(v,z))|\nonumber\\
&&=|(u,\Delta v)_{n-1}+(\Delta u,nv)_{n-1}-(\nabla w,\nabla v)+(nu,z)_{n-1}|\nonumber\\
&&=|(u,\Delta v)_{n-1}+(\Delta u,nv)_{n-1}-(w,\Delta v)+(nu,z)_{n-1}|\nonumber\\
&&\lesssim\|u\|_{0}\|v\|_{2}+\|\frac{n}{n-1}\Delta u\|_{-1}\|v\|_{1}+\|w\|_{0}\|v\|_{2}
+\|u\|_{0}\|z\|_{0}\nonumber\\
&&\lesssim(\|u\|_{0}+\|\frac{n}{n-1}\Delta u\|_{-1}+\|w\|_{0})\|(v,z)\|_{\mathbf{H}},
~\forall(u,w),(v,z)\in\mathbf{H}.\label{s2.23}
\end{eqnarray}
When $n\in W^{1,\infty}(\Omega)$,
\begin{eqnarray}
&&~~~|B((u,w),(v,z))|\nonumber\\
&&=|(u,\Delta v)_{n-1}-(\nabla u,\nabla (\frac{n}{n-1}v))+(w,\Delta v)+(nu,z)_{n-1}|\nonumber\\
&&\lesssim\|u\|_{0}\|v\|_{2}+\|u\|_{1}\|v\|_{1}+\|w\|_{0}\|v\|_{2}
+\|u\|_{0}\|z\|_{0}\nonumber\\
&&\lesssim(\|u\|_{1}+\|w\|_{0})(\|v\|_{2}+\|z\|_{1})\nonumber\\
&&=\|(u,w)\|_{\mathbf{H_{1}}}\|(v,z)\|_{\mathbf{H}},
~~~\forall(u,w)\in\mathbf{H_{1}},
\forall(v,z)\in\mathbf{H}.\label{s2.2}
\end{eqnarray}
\indent (\ref{s2.23}) and (\ref{s2.2}) tell us that for any given $(u,w)\in\mathbf{H}$ or $(u,w)\in\mathbf{H_1}$, $B((u,w),(v,z))$ is a continuous linear form on $\mathbf{H}$.\\
\indent In paper \cite{8} Monk et al. establish  the weak formulation of
(\ref{ss2.1})-(\ref{ss2.4}): find $(u,w)\in\mathbf{H}$ and
$\lambda\in \mathbb{C}$ such that
\begin{eqnarray}
&&\lambda A((u,w),(v,z))=B((u,w),(v,z)),~~~\forall(v,z)\in\mathbf{H}.\label{s2.3}
\end{eqnarray}
The number $(\lambda)^{-1}=\tau=k^2$ is transmission eigenvalue.\\
\indent The source problem associated with (\ref{s2.3}) is: For any given $(f,g)\in\mathbf{H_{1}}$, find $(\psi,\varphi)\in\mathbf{H}$ such that
\begin{eqnarray}
&&A((\psi,\varphi),(v,z))=B((f,g),(v,z)),~~~\forall(v,z)\in\mathbf{H}.\label{s2.4}
\end{eqnarray}
From Lax-Milgram theorem we know that (\ref{s2.4}) exists an unique solution, therefore we define the corresponding solution operator $T:\mathbf{H_{1}}\rightarrow\mathbf{H}$ by
\begin{eqnarray}
&&A(T(f,g),(v,z))=B((f,g),(v,z)),~~~\forall(v,z)\in\mathbf{H}.\label{s2.5}
\end{eqnarray}
\indent Then, in operator notation, the problem of
(\ref{s2.3}) is to find $(u,w)\in\mathbf{H}\backslash\{\mathbf{0}\}$
and $\lambda\in \mathbb{C}$ such that
\begin{eqnarray}
&&T(u,w)=\lambda(u,w).\label{s2.6}
\end{eqnarray}
Referring to Theorem 2.2 of \cite{13} we can deduce the following result:
\begin{lemma}
When $n\in L^{\infty}(\Omega)$ ,the operator $T:\mathbf{H}\rightarrow\mathbf{H}$ is compact,
 and when $n\in W^{1,\infty}(\Omega)$, $T:\mathbf{H_{1}}\rightarrow\mathbf{H_{1}}$ is compact.\\
 \end{lemma}
\indent The dual problem of (\ref{s2.3}) is: Find $(u^{*},w^{*})\in
\mathbf{H}\backslash\{\mathbf{0}\}$ and $\lambda^{*}\in \mathbb{C}$
such that
\begin{eqnarray}
&&\overline{\lambda^{*}}A((v,z),(u^{*},w^{*}))=B((v,z),(u^{*},w^{*})),
~~~\forall(v,z)\in \mathbf{H}.\label{s2.7}
\end{eqnarray}
Note that (\ref{s2.3}) and (\ref{s2.7}) are connected via $\lambda=\overline{\lambda^{*}}$.\\
\indent Using the same method we define corresponding operator $T^{*}:\mathbf{H_{1}}\rightarrow\mathbf{H}$ by
\begin{eqnarray}
&&A((v,z),T^{*}(f,g))=B((v,z),(f,g)),~~~\forall(v,z)\in\mathbf{H},\label{s2.8}
\end{eqnarray}
and (\ref{s2.8}) has the equivalent operator form:
\begin{eqnarray}
&&T^{*}(u,w)=\lambda^{*}(u,w).\label{s2.9}
\end{eqnarray}
\indent It's obvious that $T^{*}$ is the adjoint operator of $T$ in the sense of inner product $A(\cdot,\cdot)$, in fact
\begin{eqnarray}
&&A(T(f,g),(v,z))=B((f,g),(v,z))\nonumber\\
&&~~~~~~~~~~~~~~~~~~~~~~=A((f,g),T^{*}(v,z)),
~~~\forall(f,g),(v,z)\in\mathbf{H}.\label{s2.10}
\end{eqnarray}
\indent Let $\pi_{h}$ be a shape-regular grid of
$\Omega$ with mesh size $h$. As the same with \cite{8}, we use
$X_{h}\subset H_{0}^{2}(\Omega)$ and $Y_{h}\subset
H_{0}^{1}(\Omega)$ to compute
 the finite dimensional problem. Let $\mathbf{H_{h}}=X_{h}\times Y_{h}$ and $\mathbf{H_{h}}\subset
 \mathbf{H}$.
  The $X_{h}$ can be made up of one of the Argyris element, the Bell element, and the Bonger-Fox-Schmit elemnt (see \cite{16}). On the other hand, $Y_{h}$ can be built with bilinear Lagrange element or biquadratic Lagrange element \cite{15}.\\
\indent Then the finite element approximation of
(\ref{s2.3}) is: Find $\lambda_{h}\in \mathbb{C}$, $(u_{h},w_{h})\in
\mathbf{H_{h}}\backslash \{\mathbf{0}\}$ such that
\begin{eqnarray}
&&\lambda_{h}A((u_{h},w_{h}),(v,z))=B((u_{h},w_{h}),(v,z)),
~~~\forall(v,z)\in \mathbf{H_{h}}.\label{s2.11}
\end{eqnarray}
\indent Consider the approximate source problem: For any given $(f,g)\in \mathbf{H_{1}}$, find $(\psi_{h},\varphi_{h})\in\mathbf{H_{h}}$ such that
\begin{eqnarray}
&&A((\psi_{h},\varphi_{h}),(v,z))=B((f,g),(v,z)),
~~~\forall(v,z)\in \mathbf{H_{h}}.\label{s2.12}
\end{eqnarray}
And then , the associated solution operator
$T_{h}:\mathbf{H_{1}}\rightarrow \mathbf{H_{h}}$
satisfies
\begin{eqnarray}
&&A(T_{h}(f,g),(v,z))=B((f,g),(v,z)),
~~~\forall(v,z)\in \mathbf{H_{h}}.\label{s2.13}
\end{eqnarray}
The equation (\ref{s2.11}) has the operator form:
\begin{eqnarray}
&&T_{h}(u_{h},w_{h})=\lambda_{h}(u_{h},w_{h}).\label{s2.14}
\end{eqnarray}
\indent Using the same method we can easily give the finite element approximate problem of (\ref{s2.7}) and define the conjugated operator $T_{h}^{*}$ in the sense of inner product $A(\cdot,\cdot)$ on $\mathbf{H_h}$. We omit it here striving for conciseness.\\
\indent Now define the projection operator $P_{h}^{1}:H_{0}^{2}(\Omega)\rightarrow X_{h}$ and $P_{h}^{2}:H_{0}^{1}(\Omega)\rightarrow Y_{h}$ by
\begin{eqnarray}
&&(\Delta(u-P_{h}^{1}u),\Delta v))_{n-1}=0,
~~~\forall v\in X_{h},\label{s2.15}\\
&&(\nabla(w-P_{h}^{2}w),\nabla z)=0,
~~~\forall z\in Y_{h}.\label{s2.16}
\end{eqnarray}
Let $$P_{h}(u,w)=(P_{h}^{1}u,P_{h}^{2}w),~~~\forall(u,w)\in\mathbf{H}.$$
Then $P_{h}:\mathbf{H}\rightarrow\mathbf{H_{h}}$ and
\begin{eqnarray}
&&A((u,w)-P_{h}(u,w),(v,z))=A((u,w)-(P_{h}^{1}u,P_{h}^{2}w),(v,z))\nonumber\\
&&~~~=(\Delta(u-P_{h}^{1}u),\Delta v)_{n-1}+(\nabla(w-P_{h}^{2}w),\nabla z)\nonumber\\
&&~~~=0,~~~\forall(v,z)\in\mathbf{H_{h}}.\label{s2.17}
\end{eqnarray}
$P_{h}:\mathbf{H}\rightarrow\mathbf{H_{h}}$ is the Ritz projection.\\
\indent Utilizing (\ref{s2.17}), (\ref{s2.5}) and (\ref{s2.13}) we get that for any $(u,w)\in\mathbf{H}$,
\begin{eqnarray}
&&A(P_{h}T(u,w)-T_{h}(u,w),(v,z))\nonumber\\
&&~~~~~~~~~~=A(P_{h}T(u,w)-T(u,w),(v,z))+A(T(u,w)-T_{h}(u,w),(v,z))\nonumber\\
&&~~~~~~~~~~=0,~~~\forall(v,z)\in\mathbf{H_{h}}.\label{s2.18}
\end{eqnarray}
So
\begin{eqnarray}
&&T_{h}=P_{h}T.\label{s2.19}
\end{eqnarray}

\section{Error analysis}
\indent   Throughout this paper , we suppose the following condition $(C1)$ holds:\\
\indent $(C1)$ If $(\psi,\varphi)\in\mathbf{H}$, then as $h\rightarrow 0$,
$$\inf\limits_{(v,z)\in \mathbf{H_{h}}
}\|(\psi,\varphi)-(v,z)\|_{\mathbf{H}}\to 0.$$
 \indent
For the operators $T$ and $T_h$, we have the following important conclusions:
\begin{theorem}
Let $n \in L^{\infty}(\Omega)$, then
\begin{eqnarray}\label{s3.1}
&&\|T-T_{h}\|_{\mathbf{H}}\to 0,
\end{eqnarray}
and when $n \in W^{1,\infty}(\Omega)$, there exists
\begin{eqnarray}\label{s3.2}
&&\|T-T_{h}\|_{\mathbf{H}_{1}}\to 0.
\end{eqnarray}
\end{theorem}
\begin{proof}
 Based on Lemma 2.1, using the method which is similar with the proof of Theorem 3.1 in \cite{13} we can easily get the conclusions above.
\end{proof}
\indent As usual, we suppose $\lambda=\lambda_k$ is the $k$th eigenvalue of (\ref{s2.3}) with the algebraic multiplicity $q$ and the ascent $\alpha$ ($\lambda_k=\lambda_{k+1}=\cdots =\lambda_{k+q-1}$). $\lambda^*=\lambda_k^{*}=\overline{\lambda_k}$ is an eigenvalue of (\ref{s2.7}). Eigenvalues $\lambda_{k,h}, \cdots, \lambda_{k+q-1,h}$ of (\ref{s2.11}) will converge to $ \lambda$ due to the Theorem 3.1.\\
\indent We define $E$ as the spectral projection associated with $T$ and $\lambda_k$, range $R(E)$ is the space of generalized eigenfunctions associated with $T$ and $\lambda$. Let $E_{h}$ be the spectral projection associated with $T_h$ and the eigenvalues $\lambda_{k,h}, \cdot\cdot\cdot, \lambda_{k+q-1,h}$ and let $M_h(\lambda)$ be the space of generalized eigenfunctions associated with $T_h$ and $\lambda_{k,h}, \cdot\cdot\cdot, \lambda_{k+q-1,h}$, it's obvious that the range $R(E_{h})=M_h(\lambda)$ if $h$ is small enough. On the dual problem (\ref{s2.7}) and its finite element approximate problem, the definitions of $E^{*}$, $R(E^{*})$, $E_{h}^{*}$, $M_{h}(\lambda^*)$ and $R(E_{h}^{*})$ are also analogous to the former.\\
\indent Given two closed subspaces $V$ and $U$, let
\begin{eqnarray*}
&&\delta(V,U)=\sup\limits_{(u,w)\in V\atop\|(u,w)\|_{\mathbf{H}}=1}\inf\limits_{(v,z)\in U}\|(u,w)-(v,z)\|_{\mathbf{H}},\\
&&\theta(V,U)_{1}=\sup\limits_{(u,w)\in
V\atop\|(u,w)\|_{\mathbf{H_{1}}}=1}\inf\limits_{(v,z)\in
U}\|(u,w)-(v,z)\|_{\mathbf{H_{1}}}.
\end{eqnarray*}
and define the gaps between $R(E)$ and $R(E_{h})$ in
$\|\cdot\|_{\mathbf{H}}$ by
\begin{eqnarray*}
 \widehat{\delta}(R(E),R(E_{h}))=\max\{\delta(R(E),R(E_{h})),
 \delta(R(E_{h}),R(E))\},
\end{eqnarray*}
and in $\|\cdot\|_{\mathbf{H}_{1}}$ by
\begin{eqnarray*}
 \widehat{\theta}(R(E),R(E_{h}))_{1}=\max\{\theta(R(E),R(E_{h}))_{1},
 \theta(R(E_{h}),R(E))_{1}\}.
\end{eqnarray*}
Define
\begin{eqnarray*}
&&\varepsilon_{h}(\lambda)=\sup\limits_{(u,\omega)\in
R(E)\atop\|(u,\omega)\|_{\mathbf{H}}=1} \inf\limits_{(v,z)\in
\mathbf{H}_{h}}\|(u,\omega)-(v,z)\|_{\mathbf{H}},\\
&&\varepsilon_{h}^{*}(\lambda^{*})=\sup\limits_{(u^{*},\omega^{*})\in
R(E^{*})\atop\|(u^{*},\omega^{*})\|_{\mathbf{H}}=1}
\inf\limits_{(v,z)\in
\mathbf{H}_{h}}\|(u^{*},\omega^{*})-(v,z)\|_{\mathbf{H}}.
\end{eqnarray*}
From $(C1)$ we know that
\begin{eqnarray*}
\varepsilon_{h}(\lambda)\to 0~(h\to
0),~~~\varepsilon_{h}^{*}(\lambda^{*})\to 0~(h\to 0).
\end{eqnarray*}
\indent The following theorem has been proved by \cite{8} with the condition of $n$ is smooth. Thanks to Theorem 3.1, we know the error estimates  also hold when $n \in L^\infty(\Omega)$.
\begin{theorem}
Suppose $n \in L^{\infty}(\Omega)$, then
\begin{eqnarray}\label{s3.3}
&&\widehat{\delta}(R(E),
R(E_{h}))\lesssim\varepsilon_{h}(\lambda)\\\label{s3.4}
&&|\lambda^{-1}-(\frac{1}{q}\sum\limits_{j=1}^{q}\lambda_{j,h})^{-1}|\lesssim
\varepsilon_{h}(\lambda)\varepsilon_{h}^{*}(\lambda^{*}),\\\label{s3.5}
&&|\lambda^{-1}-\lambda_{j,h}^{-1}|\lesssim
[\varepsilon_{h}(\lambda)\varepsilon_{h}^{*}(\lambda^{*})]^{\frac{1}{\alpha}}.
\end{eqnarray}
Let $(u_{h},w_{h})$ with $\|(u_{h},w_{h})\|_{A}=1$ is eigenfunction corresponding to
$\lambda_{j,h}$ ($j=1,2,\cdots,q$), then there exists eigenfunction
$(u,w)$ corresponding to $\lambda$, such that
\begin{eqnarray}\label{s3.6}
&&\|(u_{h},w_{h})-(u,w)\|_{\mathbf{H}}\lesssim
\varepsilon_{h}(\lambda)^{\frac{1}{\alpha}}.
\end{eqnarray}
\end{theorem}
\indent The operator interpolation theory (see \cite{15}) tells us that the following condition $(C2)$ holds:\\
\indent $(C2)$ If $\psi\in H_{0}^{2}(\Omega)\cap H^{2+\sigma_{1}}(\Omega),(\sigma_{1}\in(0,2])$, then
\begin{eqnarray}
&&\inf\limits_{v\in X_{h}}\|\psi-v\|_{s}\lesssim h^{2+\sigma_{1}-s}\|\psi\|_{2+\sigma_{1}},~~~~~s=1,2\label{s3.7}
\end{eqnarray}
if $\varphi\in H_{0}^{1}(\Omega)\cap H^{1+\sigma_{2}}(\Omega)$,($\sigma_{2}\in(0,1]$ for bilinear Lagrange element, $\sigma_{2}\in(0,2]$ for biquadratic Lagrange element), then
\begin{eqnarray}
&&\inf\limits_{v\in Y_{h}}\|\varphi-v\|_{t}\lesssim h^{1+\sigma_{2}-t}\|\varphi\|_{1+\sigma_{2}},~~~~~t=0,1\label{s3.8}
\end{eqnarray}
\begin{corollary}
Suppose that $n \in L^{\infty}(\Omega), R(E), R(E^{*})\subset
\mathbf{H}\cap (H^{2+\sigma_{1}}(\Omega)\times H^{1+\sigma_{2}}(\Omega))$, and
(C2) is valid. Then (\ref{s3.3})-(\ref{s3.5}) hold with
\begin{eqnarray}
 \varepsilon_{h}(\lambda)\lesssim h^{\sigma},~~~\varepsilon_{h}^{*}(\lambda^{*})\lesssim h^{\sigma},~~~~~~\sigma=min\{\sigma_{1},\sigma_{2}\}.\label{s3.9}
\end{eqnarray}
\end{corollary}
\indent  Next, referring to \cite{13}, we use the
Aubin-Nitsche technique to discuss the error estimates in norm
$\|\cdot\|_{\mathbf{H_{1}}}$. We need the following regularity
assumption :\\
\indent $(A1)$~For any $\xi\in H^{-1}(\Omega)$, there exists
$\psi\in H^{2+r_{1}}(\Omega)$ satisfying
\begin{eqnarray*}
&&\Delta(\frac{1}{n-1}\Delta \psi)=\xi,~~~in~\Omega,\\
&&\psi=\frac{\partial \psi}{\partial \nu}=0~~~ on~\partial \Omega,
\end{eqnarray*}
\begin{eqnarray}\label{s3.10}
&&\|\psi\|_{2+r_{1}}\leq C_p \|\xi\|_{-1},~(\|\xi\|_{-1}=\sup\limits_{0\neq v\in H_{0}^{1}(\Omega)}\frac{|(\xi,v)|}{\|v\|_{1}}).
\end{eqnarray}
where $r_{1}\in(0,1]$,  $C_p$ is the prior constant dependent on the equation and $\Omega$ but independent of the right-hand side $\xi$ of the equation.\\
\indent Consider the auxiliary boundary value problem with the assumption of $n\in W^{2,\infty}(\Omega)$ (if $\partial \Omega$ is a convex polygon, $r_{1}$ can reach the value 1 (see \cite{17})):
\begin{eqnarray}\label{s3.11}
&&\Delta(\frac{1}{n-1}\Delta \phi_{1})=-\Delta
(u-P_{h}^{1}u),~~~in~\Omega,\\\label{s3.12}
 &&\phi_{1}=\frac{\partial
\phi_{1}}{\partial \nu}=0,~~~ on~\partial \Omega.
\end{eqnarray}
Assume (A1) holds, then we can deduce that
\begin{eqnarray}
&&\|\phi_{1}\|_{2+r_{1}}\lesssim\|\Delta(u-P_{h}^{1})\|_{-1}=\sup\limits_{v\in H_{0}^{1}(\Omega)\atop \|v\|_{1}=1}|(\Delta(u-P_{h}^{1}u),v)|\nonumber\\
&&~~~=\sup\limits_{v\in H_{0}^{1}(\Omega)\atop \|v\|_{1}=1}|(\nabla(u-P_{h}^{1}v),\nabla v)|\lesssim\|u-P_{h}^{1}u\|_{1}\|v\|_{1}\nonumber\\
&&~~~\lesssim\|u-P_{h}^{1}u\|_{1}.\label{s3.13}
\end{eqnarray}
The weak form of (\ref{s3.11})-(\ref{s3.12}) is $$(\Delta v,\Delta
\phi_{1})_{n-1}=(\nabla v,\nabla (u-P_{h}^{1}u)),~~~~~\forall v\in
H_{0}^{2}(\Omega).$$ Let $v=u-P_{h}^{1}u$, take the definition
$P_{h}^{1}$, then
$$(\Delta(u-P_{h}^{1}u),\Delta(\phi_{1}-P_{h}^{1}\phi_{1}))_{n-1}
=(\nabla(u-P_{h}^{1}u),\nabla(u-P_{h}^{1}u)).$$
Using (\ref{s3.7}) and (\ref{s3.13}) we get
\begin{eqnarray*}
&&\|\nabla(u-P_{h}^{1}u)\|_{0}^{2}\lesssim
\|\Delta(u-P_{h}^{1}u)\|_{0}\|\Delta(\phi_{1}-P_{h}^{1}\phi_{1})\|_{0}\nonumber\\
&&~~~\lesssim\|\Delta(u-P_{h}^{1}u)\|_{0}\|\phi_{1}-P_{h}^{1}\phi_{1}\|_{2}
\lesssim\|\Delta(u-P_{h}^{1}u)\|_{0}h^{r_{1}}\|\phi_{1}\|_{2+r_{1}}\nonumber\\
&&~~~\lesssim
h^{r_{1}}\|\Delta(u-P_{h}^{1}u)\|_{0}\|u-P_{h}^{1}u\|_{1}.
\end{eqnarray*}
So
\begin{eqnarray}
&&\|u-P_{h}^{1}u\|_{1}\lesssim h^{r_{1}}\|u-P_{h}^{1}u\|_{2}.\label{s3.14}
\end{eqnarray}
\indent We also need the regularity assumption (see
\cite{18}):\\
\indent $(A2)$~ For any $f\in L^{2}(\Omega)$, there exist
$\varphi\in H^{1+r_{2}}$ which satisfies
\begin{eqnarray}
&&\Delta \varphi=f,~~~~~in~\Omega,\nonumber\\
&&\varphi=0,~~~~~on~\partial\Omega,\nonumber\\
&&\|\varphi\|_{1+r_{2}}\leq C_p\|f\|_{0},\label{s3.15}
\end{eqnarray}
where $r_{2}\in(0,1]$,  $C_p$ is the prior constant dependent on the equation and $\Omega$.\\
\indent The auxiliary boundary value problem is:
\begin{eqnarray}
&&\Delta \phi_{2}=w-P_{h}^{2}w,~~~in~\Omega,\label{s3.16}\\
&&\phi_{2}=0,~~~on~\partial\Omega.\label{s3.17}
\end{eqnarray}
Assume (A2) holds, then
\begin{eqnarray}
&&\|\phi_{2}\|_{1+r_{2}}\lesssim \|w-P_{h}^{2}w\|_{0}.\label{s3.18}
\end{eqnarray}
The associated weak form of (\ref{s3.16})-(\ref{s3.17}):
$$(\nabla \phi_{2},\nabla z)=(w-P_{h}^{2},z),~~~\forall z\in H_{0}^{1}(\Omega).$$
Let $z=w-P_{h}^{2}w$, and use the definition $P_{h}^{2}$ , then
$$(\nabla(\phi_{2}-P_{h}^{2}\phi_{2}),\nabla(w-P_{h}^{2}w))
=(w-P_{h}^{2},w-P_{h}^{2})$$
Combining (\ref{s3.8}) with (\ref{s3.18}), we deduce that
\begin{eqnarray*}
&&\|(w-P_{h}^{2}w)\|_{0}^{2}\lesssim
\|\nabla(w-P_{h}^{2}w)\|_{0}\|\nabla(\phi_{2}-P_{h}^{2}\phi_{2})\|_{0}\nonumber\\
&&~~~\lesssim\|(w-P_{h}^{2}w)\|_{1}\|\phi_{2}-P_{h}^{2}\phi_{2}\|_{1}
\lesssim\|w-P_{h}^{2}w\|_{1}h^{r_{2}}\|\phi_{2}\|_{1+r_{2}}\nonumber\\
&&~~~\lesssim h^{r_{2}}\|w-P_{h}^{2}w\|_{0}\|w-P_{h}^{2}w\|_{1}.
\end{eqnarray*}
So
\begin{eqnarray}
&&\|w-P_{h}^{2}w\|_{0}\lesssim h^{r_{2}}\|w-P_{h}^{2}w\|_{1}.\label{s3.19}
\end{eqnarray}
\indent Now we have the following lemma.
\begin{lemma}
Suppose that $(C1)$, $(C2)$, $(A1)$ and $(A2)$ are valid, then for $(u,w)\in\mathbf{H}$,
\begin{eqnarray}
&&\|(u,w)-P_{h}(u,w)\|_{\mathbf{H_{1}}}\lesssim h^{r}\|(u,w)-P_{h}(u,w)\|_{\mathbf{H}},\label{s3.20}
\end{eqnarray}
where $r=min\{r_{1},r_{2}\}$.
\end{lemma}
\begin{proof}
Combining (\ref{s3.14}) with (\ref{s3.19}),we know
\begin{eqnarray*}
&&\|(u,w)-P_{h}(u,w)\|_{\mathbf{H_{1}}}
=\|(u-P_{h}^{1}u,w-P_{h}^{2}w)\|_{\mathbf{H_{1}}}\\
&&~~~~~~=\|u-P_{h}^{1}u\|_{1}+\|w-P_{h}^{2}w\|_{0}\lesssim h^{r_{1}}\|u-P_{h}^{1}u\|_{2}
+h^{r_{2}}\|w-P_{h}^{2}w\|_{1}\\
&&~~~~~~\lesssim h^{r}(\|u-P_{h}^{1}u\|_{2}+\|w-P_{h}^{2}w\|_{1})
=h^{r}\|(u,w)-P_{h}(u,w)\|_{\mathbf{H}}.
\end{eqnarray*}
The proof is complete.
\end{proof}
\begin{theorem}
Suppose that $(C1)$, $(C2)$, $(A1)$ and $(A2)$ are valid, and $n\in W^{2,\infty}(\Omega),r\in(0,1]$. Then
\begin{eqnarray}\label{s3.21}
\widehat{\theta}(R(E), R(E_{h}))_{1}
 \lesssim h^{r}\varepsilon_{h}(\lambda).
\end{eqnarray}
Let $\lambda_{h}$ be the eigenvalue of (\ref{s2.11}) which converges to $\lambda$ ,and let $(u_{h},w_{h})$ with $\|(u_{h},w_{h})\|_{A}=1$ be an eigenfunction associated with $\lambda_{h}$ ,then there exists eigenfunction $(u,w)$ corresponding to $\lambda$, such that
\begin{eqnarray}\label{s3.22}
\|(u_{h},\omega_{h})-(u,\omega)\|_{\mathbf{H}_{1}}\lesssim
(h^{r}\varepsilon_{h}(\lambda))^{\frac{1}{\alpha}}.
\end{eqnarray}
\end{theorem}
\begin{proof}
In $R(E)$ the norm $\|\cdot\|_{\mathbf{H}}$ is equivalent to the norm $\|\cdot\|_{\mathbf{H_{1}}}$, and $TR(E)\subset R(E)$, referring to (\ref{s3.20}) we know that
\begin{eqnarray*}
&& \|(T-T_{h})|_{R(E)}\|_{\mathbf{H}_{1}}=
\sup\limits_{(u,w)\in R(E)}
\frac{\|T(u,w)-T_{h}(u,w)\|_{\mathbf{H}_{1}}}{\|(u,w)\|_{\mathbf{H}_{1}}}\\
&&~~~\lesssim \sup\limits_{(u,w)\in R(E)}
\frac{\|T(u,w)-P_{h}T(u,w)\|_{\mathbf{H}_{1}}}{\|(u,w)\|_{\mathbf{H}}}\\
&&~~~\lesssim h^{r}\sup\limits_{(u,w)\in R(E)}
\frac{\|T(u,w)-P_{h}T(u,w)\|_{\mathbf{H}}}{\|(u,w)\|_{\mathbf{H}}}\\
&&~~~=h^{r}\sup\limits_{(u,w)\in R(E)}
\frac{\|T(u,w)-P_{h}T(u,w)\|_{\mathbf{H}}}{\|T(u,w)\|_{\mathbf{H}}}
\frac{\|T(u,w)\|_{\mathbf{H}}}{\|(u,w)\|_{\mathbf{H}}}\\
&&~~~ \lesssim h^{r}\varepsilon_{h}(\lambda).
\end{eqnarray*}
Combining Theorem 3.1 with Theorem 7.1 in \cite{14}, we get
\begin{eqnarray*}
\widehat{\theta}(R(E), R(E_{h}))_{1}\lesssim
\|(T-T_{h})|_{R(E)}\|_{\mathbf{H}_{1}}\lesssim
h^{r}\varepsilon_{h}(\lambda).
\end{eqnarray*}
And from Theorem 7.4 in \cite{14} we have
\begin{eqnarray*}
\|u_{h}-u\|_{\mathbf{H}_{1}}\lesssim
\|(T-T_{h})|_{R(E)}\|_{\mathbf{H}_{1}}^{\frac{1}{\alpha}}\lesssim
(h^{r}\varepsilon_{h}(\lambda))^{\frac{1}{\alpha}}.
\end{eqnarray*}
So we complete the proof.
\end{proof}
\indent If we take the same method in this section, we can deduce the error estimates of finite element approximation for the dual problem (\ref{s2.7}), the following results can be easily get
\begin{eqnarray}\label{s3.23}
&&\|(u_{h}^{*},\omega_{h}^{*})-(u^{*},\omega^{*})\|_{\mathbf{H}}\lesssim
\varepsilon_{h}^{*}(\lambda^{*})^{\frac{1}{\alpha}},\\\label{s3.24}
&&\|(u_{h}^{*},\omega_{h}^{*})-(u^{*},\omega^{*})\|_{\mathbf{H}_{1}}\lesssim
(h^{r}\varepsilon_{h}^{*}(\lambda^{*}))^{\frac{1}{\alpha}}.
\end{eqnarray}

\section{A mixed-element two-grid discretization}
\indent Now we use the two-grid discretization to deal with the transmission eigenvalues problem and consider its error estimates.\\
\indent{\bf Definition 4.1.}~~ $\forall$ $(v,z),(v^{*},z^{*})\in \mathbf{H}
$, $B((v,z),(v^{*},z^{*}))\neq 0$, define
\begin{eqnarray*}
\frac{A((v,z),(v^{*},z^{*}))}{B((v,z),(v^{*},z^{*}))}
\end{eqnarray*}
as the generalized Rayleigh quotient of $(v,z)$ and $(v^{*},z^{*})$.\\
\indent Now take three steps to achieve the two-grid
discretization.\\ \indent {\bf Step 1.}~ Solve (\ref{s2.11}) on a
coarse grid $\pi_{H}$: Find $\lambda_{H}\in \mathbb{C},
(u_{H},w_{H})\in \mathbf{H}_{H}$ such that
$\|(u_{H},w_{H})\|_{\mathbf{A}}=1$ and
\begin{eqnarray*}
\lambda_{H}A((u_{H},w_{H}),(v,z)) =B((u_{H},w_{H}),(v,z)),~~~\forall (v,z)\in \mathbf{H}_{H},
\end{eqnarray*}
and find $(u_{H}^{*}, w_{H}^{*})\in R(E_{H}^*)$ with
$\|(u_{H}^{*},w_{H}^{*})\|_{\mathbf{A}}=1$ such that
$|B((u_{H},w_{H}), (u_{H}^{*},w_{H}^{*}))|$ has a positive
lower bound uniformly with respect to $H$.\\
\indent {\bf Step 2.} Solve two linear boundary value problems on a
fine grid $\pi_{h}$: Find $(u^h,w^{h})\in \mathbf{H}_{h}$ such
that
\begin{eqnarray*}
\lambda_{H}A((u^h,w^{h}),
(v,z))=B((u_{H},w_{H}),(v,z)),~~~\forall (v,z)\in
\mathbf{H}_{h},
\end{eqnarray*}
and find $(u^{h*},w^{h*})\in \mathbf{H}_{h}$ such that
\begin{eqnarray*}
\lambda_{H}A((v,z), (u^{h*},w^{h*}))=B((v,z),(u_{H}^{*},w_{H}^{*})),~~~\forall (v,z)\in
\mathbf{H}_{h}.
\end{eqnarray*}
\indent {\bf Step 3.} Compute the generalized Rayleigh quotient
\begin{eqnarray*}
 (\lambda^{h})^{-1}=\frac{A((u^h,w^{h}),(u^{h*},w^{h*}))}{ B((u^h,w^{h}),(u^{h*},w^{h*}))}.
\end{eqnarray*}

\indent Thanks to \cite{13} we get the following lemma.
\begin{lemma}
 Let
$(u_{H}^{-},w_{H}^{-})$ be the orthogonal projection of
$(u_{H},w_{H})$ to $R(E_H^{*})$ in the sense of inner product
$A(\cdot,\cdot)$, and let
\begin{eqnarray}\label{s4.1}
(u_{H}^{*},w_{H}^{*})=\frac{(u_{H}^{-},w_{H}^{-})}{\|(u_{H}^{-},w_{H}^{-})\|_{A}}.
\end{eqnarray}
Then $|B((u_{H},w_{H}), (u_{H}^{*},w_{H}^{*}))|$ has a positive
lower bound uniformly with respect to $H$.
\end{lemma}
And referring to \cite{19,13,14}, we have the lemma as follow:
\begin{lemma}
Let $(\lambda,u,w)$ and $(\lambda^{*},u^{*},w^{*})$ be the
eigenpair of (\ref{s2.3}) and (\ref{s2.7}),
respectively. Then,
$\forall (v,z),(v^{*},z^{*})\in \mathbf{H}$, $B((v,z),
(v^{*},z^{*}))\not=0$, the generalized Rayleigh quotient satisfies
\begin{eqnarray}\label{s4.2}
&&\frac{A((v,z),(v^{*},z^{*}))}{B((v,z),(v^{*},z^{*}))}-\lambda^{-1}
=\frac{A((v,z)-(u,w),(v^{*},z^{*})-(u^{*},w^{*}))}{B((v,z),(v^{*},z^{*}))}
\nonumber\\
&&~~~~~~-\lambda^{-1}
\frac{B((v,z)-(u,w),(v^{*},z^{*})-(u^{*},w^{*}))}{B((v,z),(v^{*},z^{*}))}.
\end{eqnarray}
\end{lemma}

\begin{theorem}
Let $\lambda_{H}, (u_{H},w_{H}), (u_{H}^{*},w_{H}^{*}),
\lambda^{h}, (u^{h},w^{h}), (u^{h*},w^{h*})$ be the numerical eigenpairs obtained by the procedure of two-grid discretization and let $\lambda$ be the eigenvalue of (\ref{s2.3}) which is approximated by $\lambda_{H}$. Assume that the ascents of both $\lambda$ and $\lambda_H$ are equal to $1$, and $(C1)$, $(C2)$, $(A1)$, $(A2)$ are valid, $n \in W^{2,\infty}(\Omega)$. Then there exists  $(u,w)\in R(E)$ and
 $(u^{*},w^{*})\in R(E^{*})$ such that when $H$ is
properly small there hold
\begin{eqnarray}\label{s4.3}
&&\|(u^{h},w^{h})-(u,w)\|_{\mathbf{H}}\lesssim
H^{r}\varepsilon_{H}(\lambda)
+\varepsilon_{h}(\lambda),\\\label{s4.4}
&&\|(u^{h*},w^{h*})-(u^{*},w^{*})\|_{\mathbf{H}}\lesssim
H^{r}\varepsilon_{H}^{*}(\lambda^{*})
+\varepsilon_{h}^{*}(\lambda),\\\label{s4.5}
&& |(\lambda^{h})^{-1}-\lambda^{-1} | \lesssim
\{H^{r}\varepsilon_{H}(\lambda) +\varepsilon_{h}(\lambda)\}
\{H^{r}\varepsilon_{H}^{*}(\lambda^{*})
+\varepsilon_{h}^{*}(\lambda^{*})\},
\end{eqnarray}
where $r\in(0,1]$.
\end{theorem}
\begin{proof}
 From (\ref{s2.13}) and step 2 of the procedure of two-grid discretization we have $(u^{h},w^{h})=(\lambda_{H})^{-1}T_{h}(u_{H},w_{H})$.
 Regarding $(\lambda_{H})^{-1}(u_{H},w_{H})\in R(E_{H})$ we can find $\lambda^{-1}(u,w)\in R(E)$ such that
 \begin{eqnarray*}
 &&|(\lambda_{H})^{-1}|\|(u_{H},w_{H})\|_{\mathbf{H_{1}}}\|\frac{(\lambda_{H})^{-1}(u_{H},w_{H})}{|(\lambda_{H})^{-1}|\|(u_{H},w_{H})\|_\mathbf{H_{1}}}
 -\frac{\lambda^{-1}(u,w)}{|(\lambda_{H})^{-1}|\|(u_{H},w_{H})\|_\mathbf{H_{1}}}\|_{\mathbf{H_{1}}}\nonumber\\
 &&~~~~~~=|(\lambda_{H})^{-1}|\|(u_{H},w_{H})\|_\mathbf{H_{1}}\inf\limits_{(v,z)\in R(E)}\|\frac{(\lambda_{H})^{-1}(u_{H},w_{H})}{|(\lambda_{H})^{-1}|\|(u_{H},w_{H})\|_\mathbf{H_{1}}}-(v,z)\|_\mathbf{H_{1}}\nonumber\\
 &&~~~~~~\lesssim \widehat{\theta}(R(E),R(E_{h}))_{1}.
 \end{eqnarray*}
 Note that $(u,w)$ is also the eigenfunction in $R(E)$ which leads that $(u,w)=\lambda^{-1} T(u,w)$. The definition of the continuous linear operator $T_{h}:\mathbf{H_{1}}\rightarrow \mathbf{H_{h}}$ implies us
\begin{eqnarray*}
\|(\lambda_{H})^{-1}T_{h}(u_{H},w_{H})-\lambda^{-1} T_{h}
(u,w)\|_{\mathbf{H}}\lesssim
\|(\lambda_{H})^{-1}(u_{H},w_{H})-\lambda^{-1}
(u,w)\|_{\mathbf{H}_{1}}.\end{eqnarray*}
 Then using Theorem 3.5 we can derive that
 \begin{eqnarray*}
&&\|(u^{h},w^{h})-(u,w)\|_{\mathbf{H}}=\|(\lambda_{H})^{-1}T_{h}(u_{H},w_{H})
-\lambda^{-1} T(u,w)\|_{\mathbf{H}}\\
&&~~~\leq \|(\lambda_{H})^{-1}T_{h}(u_{H},w_{H})-\lambda^{-1} T_{h}
(u,w)\|_{\mathbf{H}} +\|\lambda^{-1} T_{h} (u,w)-\lambda^{-1} T(u,w) \|_{\mathbf{H}}
\\
&&~~~\lesssim \|(\lambda_{H})^{-1}(u_{H},w_{H})-\lambda^{-1}
(u,w)\|_{\mathbf{H_{1}}} +|\lambda^{-1}|\|P_{h}T(u,w)-T
(u,w) \|_{\mathbf{H}}
\\
&&~~~\lesssim \widehat{\theta}(R(E),R(E_{h}))_{1}+\varepsilon_{h}(\lambda)\\
&&~~~\lesssim H^{r}\varepsilon_{H}(\lambda)
+\varepsilon_{h}(\lambda).
\end{eqnarray*}
Now we complete the proof of (\ref{s4.3}), and using the same method we can prove (\ref{s4.4}).
From (\ref{s4.2}), we have
\begin{eqnarray}
&&\nonumber|(\lambda^{h})^{-1}-\lambda^{-1}|=|\frac{A((u^{h},w^{h})-(u,w),(u^{h*},w^{h*})-(u^{*},w^{*}))}
{B((u^{h},w^{h}),(u^{h*},w^{h*}))}\\
&&~~~~~~~~~~~~~~~~~~~~~~~-\lambda^{-1}
\frac{B((u^{h},w^{h})-(u,w),(u^{h*},w^{h*})-(u^{*},w^{*}))}
{B((u^{h},w^{h}),(u^{h*},w^{h*}))}|\nonumber\\
&&~~~~~~~~~~~~~~~~~~~\lesssim |\frac{A((u^{h},w^{h})-(u,w),(u^{h*},w^{h*})-(u^{*},w^{*}))}
{B((u^{h},w^{h}),(u^{h*},w^{h*}))}|\nonumber\\
&&~~~~~~~~~~~~~~~~~~~~~~~+|\lambda^{-1}\frac{B((u^{h},w^{h})-(u,w),(u^{h*},w^{h*})-(u^{*},w^{*}))}
{B((u^{h},w^{h}),(u^{h*},w^{h*}))}|\nonumber\\
&&~~~~~~~~~~~~~~~~~~~\lesssim |\frac{\|(u^{h},w^{h})-(u,w)\|_{A}\|(u^{h*},w^{h*})-(u^{*},w^{*})\|_{A}}
{B((u^{h},w^{h}),(u^{h*},w^{h*}))}|\nonumber\\
&&~~~~~~~~~~~~~~~~~~~~~~~+|\frac{\|(u^{h},w^{h})-(u,w)\|_{\mathbf{H_{1}}}
\|(u^{h*},w^{h*})-(u^{*},w^{*})\|_{\mathbf{H}}}
{B((u^{h},w^{h}),(u^{h*},w^{h*}))}|\nonumber\\
&&~~~~~~~~~~~~~~~~~~~\lesssim |\frac{\|(u^{h},w^{h})-(u,w)\|_{\mathbf{H}}
\|(u^{h*},w^{h*})-(u^{*},w^{*})\|_{\mathbf{H}}}
{B((u^{h},w^{h}),(u^{h*},w^{h*}))}|,\nonumber
\end{eqnarray}
\begin{eqnarray*}
&&B((u^{h},w^{h}),(u^{h*},w^{h*}))
=B((u^{h},w^{h}),(u^{h*},w^{h*}))-B((u,w),(u^{*},w^{*}))\\
&&~~~~~~+B((u,w),(u^{*},w^{*})) -B((u_{H},w_{H}),
(u_{H}^*,w_{H}^{*})) +B((u_{H},w_{H}),
(u_{H}^*,w_{H}^{*})).
\end{eqnarray*}
We know that $(u^{h},w^{h})$ and $(u_{H},w_{H})$ approximate the same eigenfunction $(u,w)$, $(u^{h*},w^{h*})$ and $(u_{H}^{*},w_{H}^{*})$ approximate the same adjoint eigenfunction $(u^{*},w^{*})$, and $|B((u_{H},w_{H}),(u_{H}^*,w_{H}^{*}))|$ has a positive lower bound uniformly with respect to $H$.
So
\begin{eqnarray*}
&&|(\lambda^{h})^{-1}-\lambda^{-1}|\lesssim
|\frac{\|(u^{h},w^{h})-(u,w)\|_{\mathbf{H}}
\|(u^{h*},w^{h*})-(u^{*},w^{*})\|_{\mathbf{H}}}
{B((u^{h},w^{h}),(u^{h*},w^{h*}))}|\nonumber\\
&&~~~~~~~~~~~~~~~~~~~\lesssim\|(u^{h},w^{h})-(u,w)\|_{\mathbf{H}}
\|(u^{h*},w^{h*})-(u^{*},w^{*})\|_{\mathbf{H}}.
\end{eqnarray*}
Referring to (\ref{s4.3}) and (\ref{s4.4}) we get (\ref{s4.5}).
\end{proof}
\begin{corollary}
 Assume that the conditions of Theorem 4.3 hold,
and $R(E), R(E^{*})\subset \mathbf{H}\cap (H^{2+\sigma_{1}}(\Omega)\times
H^{2+\sigma_{2}}(\Omega))$. Then (\ref{s4.3})-(\ref{s4.5}) hold
with
\begin{eqnarray}\label{s4.6}
\varepsilon_{H}(\lambda)\leq H^{\sigma},\varepsilon^{*}_{H}(\lambda)\leq H^{\sigma},\varepsilon_{h}(\lambda)\leq h^{\sigma},\varepsilon^{*}_{h}(\lambda)\leq h^{\sigma},~\sigma=min\{\sigma_{1},\sigma_{2}\}.
\end{eqnarray}
\end{corollary}

\section{Numerical results}

\indent Now we show some examples in the last section. The numerical results of finite element discretization and two-grid discretization are both presented in our experiment. In this section, $X_{h}\in H_{0}^{2}(\Omega)$ consists of BFS element, $Y_{h}\in H_{0}^{1}(\Omega)$ consists of biquadratic Lagrange element. Let $\{\phi_{i}\}_{i=1}^{N_{h}}$ be the basis for $X_{h}$ and $\{\psi_{i}\}_{i=1}^{M_{h}}$ be the basis for $Y_{h}$. We define the following matrices
\begin{center} \footnotesize
\begin{tabular}{lllll}\hline
Matrix&Dimension&Definition\\\hline
$A^{1}$&$N_h\times N_h$&$A^{1}_{ij}=(\Delta\phi_{j},\Delta\phi_{i})_{n-1}$\\
$A^{2}$&$M_h\times M_h$&$A^{2}_{ij}=(\nabla\psi_{j},\nabla\psi_{i})$\\
$S^{1}$&$N_h\times N_h$&$S^{1}_{ij}=(\phi_{j},\Delta\phi_i)_{n-1}$\\
$S^{2}$&$N_h\times N_h$&$S^{2}_{ij}=(\Delta\phi_{j},n\phi_i)_{n-1}$\\
$R$&$N_h\times M_h$&$R_{ij}=(\nabla\psi_j,\nabla\phi_i)$\\
$M$&$M_h\times N_h$&$M_{ij}=(n\phi_j,\psi_i)_{n-1}$\\
\hline
\end{tabular}
\end{center}
\indent Then the finite element approximation
(\ref{s2.11}) and its dual conjugated problem can be written as
follows (see \cite{8}):
\begin{eqnarray} \label{s5.1}
\lambda_h\left(
\begin{array}{lcr}
A^1&0\\
0&A^2
\end{array}
\right) \left (
\begin{array}{lcr}
\mathbf{u}\\
\mathbf{w}
\end{array}
\right)=  -\left(
\begin{array}{lcr}%I_{N_{h}}
S^1+S^2&-R\\
M&0
\end{array}
\right) \left (
\begin{array}{lcr}
\mathbf{u}\\
\mathbf{w}
\end{array}
\right),
\end{eqnarray}
\begin{eqnarray} \label{s5.2}
\lambda_h\left(
\begin{array}{lcr}
A^1&0\\
0&A^2
\end{array}
\right)^{\mathbf{T}} \left (
\begin{array}{lcr}
\mathbf{u}^*\\
\mathbf{w}^*
\end{array}
\right)=  -\left(
\begin{array}{lcr}%I_{N_{h}}
S^1+S^2&-R\\
M&0
\end{array}
\right)^{\mathbf{T}} \left (
\begin{array}{lcr}
\mathbf{u}^*\\
\mathbf{w}^*
\end{array}
\right),
\end{eqnarray}
where $\mathbf{u}=(u_{1},\cdot \cdot \cdot,u_{N_{h}})^T$ and $\mathbf{w}=(w_{1},\cdot \cdot \cdot,w_{M_{h}})^T$ such that $u_h=\sum_{i=1}^{N_h}u_i\phi_i$ and $w_h=\sum_{i=1}^{M_h}w_i\psi_i$($\mathbf{u}^*$ and $\mathbf{w}^*$ are similar to $\mathbf{u}$ and $\mathbf{w}$). We use the Matlab eigs command to compute the eigenvalues and the eigenfunctions.\\
\indent As for two-grid discretization, we assemble the following matrices:
  \begin{eqnarray} \label{s5.3}
\lambda_H\left(
\begin{array}{lcr}
A^1&0\\
0&A^2
\end{array}
\right) \left (
\begin{array}{lcr}
\mathbf{u}^h\\
\mathbf{w}^h
\end{array}
\right)=  -\left(
\begin{array}{lcr}%I_{N_{h}}
S^1+S^2&-R\\
M&0
\end{array}
\right) \left (
\begin{array}{lcr}
\mathbf{u}_H\\
\mathbf{w}_H
\end{array}
\right),
\end{eqnarray}
\begin{eqnarray} \label{s5.4}
\lambda_H\left(
\begin{array}{lcr}
A^1&0\\
0&A^2
\end{array}
\right)^{\mathbf{T}} \left (
\begin{array}{lcr}
\mathbf{u}^{h*}\\
\mathbf{w}^{h*}
\end{array}
\right)=  -\left(
\begin{array}{lcr}%I_{N_{h}}
S^1+S^2&-R\\
M&0
\end{array}
\right)^{\mathbf{T}} \left (
\begin{array}{lcr}
\mathbf{u}_{H}^*\\
\mathbf{w}_{H}^*
\end{array}
\right).
\end{eqnarray}
$\lambda_H$, $(\mathbf{u}_{H},\mathbf{w}_{H})$ and $(\mathbf{u}_{H}^*,\mathbf{w}_{H}^*)$ are got via Step 1 of two-grid discretization, $(\mathbf{u}^{h},\mathbf{w}^{h})$ and $(\mathbf{u}^{h*},\mathbf{w}^{h*})$ are the solutions of boundary value problems on a fine grid $\pi_h$ respectively.\\
\indent The final purpose is to find the scalar $k\in \mathbb{C}$, therefore we use some symbols marking the values we get from our numerical experiment:\\
 $k_{j,H}=\frac{1}{\sqrt{\lambda_{j,H}}}$: the $jth$ eigenvalue of (\ref{s2.11}) on $\pi_H$,\\
 $k_{j,h}=\frac{1}{\sqrt{\lambda_{j,h}}}$: the $jth$ eigenvalue of (\ref{s2.11}) on $\pi_h$,\\
 $k_{j}^h=\frac{1}{\sqrt{\lambda_{j}^h}}$: the $jth$ eigenvalue via the procedure of two-grid discretization.\\
\indent  We implement all computations using MATLAB 2012a on a Dell
notebook PC with 16G memory. Our program is compiled under the
package of iFEM \cite{chen}

 \indent{\bf Example 1.}~~We chose the unit square $\Omega=(-\frac{1}{2},\frac{1}{2})\times (-\frac{1}{2},\frac{1}{2})$ and consider the function $n(x)=16$ and $n(x)=8+x_1-x_2$ since we can then compare the results computed here with the other literature. In addition, we add a computing example $n(x)=8+4|x|$ in the unit square. \\

 \begin{table}
\caption{The eigenvalues on the unit square,
$n=16$.}
\begin{center} \footnotesize
\begin{tabular}{ccccccc}\hline
$j$&$H$&$h$&$k_{j,H}$&$k_{j,h}$&$k_{j}^{h}$\\\hline
1&  $\frac{\sqrt2}{4}$& $\frac{\sqrt2}{16}$&    1.8853376219&   1.8796196028&  1.8796663603&\\
1&  $\frac{\sqrt2}{8}$& $\frac{\sqrt2}{32}$&    1.8800198464&   1.8795929802&   1.8795931892&\\
1&  $\frac{\sqrt2}{16}$& $\frac{\sqrt2}{64}$&    1.8796196028&   1.8795912869&   1.8795912878&\\
1&  $\frac{\sqrt2}{16}$&    $\frac{\sqrt2}{128}$&   1.8796196028&   1.8795911807&   1.8795911816&\\
1&  $\frac{\sqrt2}{32}$&    $\frac{\sqrt2}{256}$&   1.8795929802&   1.8795911742&   1.8795911804&\\
2,3&  $\frac{\sqrt2}{4}$& $\frac{\sqrt2}{16}$&    2.4663673974&   2.4443616792&   2.4452943055&\\
2,3&  $\frac{\sqrt2}{8}$& $\frac{\sqrt2}{32}$&    2.4461052685&   2.4442441026&       2.4442473340&\\
2,3&  $\frac{\sqrt2}{16}$&    $\frac{\sqrt2}{64}$&   2.4443616792&   2.4442366022&    2.4442366150&\\
2,3&  $\frac{\sqrt2}{16}$&    $\frac{\sqrt2}{128}$&   2.4443616792&   2.4442361308&   2.4442361438&\\
2,3&  $\frac{\sqrt2}{32}$&    $\frac{\sqrt2}{256}$&   2.4442441026&   2.4442361014&   2.4442361033\\
4&  $\frac{\sqrt2}{4}$& $\frac{\sqrt2}{16}$&    2.8833995402&   2.8665486898&   2.8675716128&\\
4&  $\frac{\sqrt2}{8}$& $\frac{\sqrt2}{32}$&    2.8680026392&   2.8664462104&   2.8664512778&\\
4&  $\frac{\sqrt2}{16}$&    $\frac{\sqrt2}{64}$&   2.8665486898&   2.8664395581&   2.8664395799&\\
4&  $\frac{\sqrt2}{16}$&    $\frac{\sqrt2}{128}$&   2.8665486898&   2.8664391379&   2.8664391599&\\
4&  $\frac{\sqrt2}{32}$&    $\frac{\sqrt2}{256}$&   2.8664462104&   2.8664391116&   2.8664391127&\\
\\\hline
\end{tabular}
\end{center}
\end{table}

\begin{table}
\caption{The eigenvalues on the unit square,
$n=8+x_1-x_2$.}
\begin{center} \footnotesize
\begin{tabular}{ccccccc}\hline
$j$&$H$&$h$&$k_{j,H}$&$k_{j,h}$&$k_{j}^{h}$\\\hline
1&  $\frac{\sqrt2}4$&   $\frac{\sqrt2}{16}$&    2.8377800453&   2.8222600331&   2.8223680171\\
1&  $\frac{\sqrt2}{8}$& $\frac{\sqrt2}{32}$&    2.8232753032&   2.8221938157&   2.8221943214\\
1&  $\frac{\sqrt2}{16}$& $\frac{\sqrt2}{64}$&    2.8222600331&   2.8221896223&   2.8221896244\\
1&  $\frac{\sqrt2}{16}$&    $\frac{\sqrt2}{128}$&   2.8222600331&   2.8221893584&   2.8221893606\\
1&  $\frac{\sqrt2}{32}$&    $\frac{\sqrt2}{256}$&   2.8221938157&   2.8221893421&   2.8221893420\\
2&  $\frac{\sqrt2}{4}$& $\frac{\sqrt2}{16}$&    3.5767095970&   3.5389113678&   3.5413242169\\
2&  $\frac{\sqrt2}{8}$& $\frac{\sqrt2}{32}$&    3.5418993485&   3.5387103645&   3.5387175444\\
2&  $\frac{\sqrt2}{16}$&    $\frac{\sqrt2}{64}$&   3.5389113678&   3.5386975540&   3.5235712458\\
2&  $\frac{\sqrt2}{16}$&    $\frac{\sqrt2}{128}$&   3.5389113678&   3.5386967489&   3.5386967774\\
2&  $\frac{\sqrt2}{32}$&    $\frac{\sqrt2}{256}$&   3.5387103645&   3.5386966983&   3.5386966986\\
5,6&    $\frac{\sqrt2}{4}$& $\frac{\sqrt2}{16}$&    4.4421169489 &  4.4966278055 &  4.4975272996 \\
 &  &   &                                      $\pm$0.8175028917$i$&   $\pm$0.8718292888$i$&   $\pm$0.8655562063$i$\\
5,6&    $\frac{\sqrt2}{8}$& $\frac{\sqrt2}{32}$&    4.4966139921 &   4.4965576676 & 4.4964837836 \\
&&&                                             $\pm$0.8766783148$i$&     $\pm$0.8715036461$i$&  $\pm$0.8715549008$i$\\
5,6&    $\frac{\sqrt2}{16}$&    $\frac{\sqrt2}{64}$&   4.4966278055 &   4.4965523241&     4.4965520112 \\
&&&                                               $\pm$0.8718292888$i$&   $\pm$0.8714831489$i$&     $\pm$0.8714833311$i$\\
5,6&    $\frac{\sqrt2}{16}$&    $\frac{\sqrt2}{128}$&   4.4966278055 &   4.4965519778&     4.4965516628 \\
&&&                                               $\pm$0.8718292888$i$&   $\pm$0.8714818661$i$&     $\pm$0.8714820491$i$\\
5,6&    $\frac{\sqrt2}{32}$&    $\frac{\sqrt2}{256}$&   4.4965576676 &   4.4965519559&       4.4965519547\\
&&&                                                    $\pm$0.8715036461$i$&   $\pm$0.8714817861$i$~~&  $\pm$0.8714817865$i$\\
\\\hline
\end{tabular}
\end{center}
\end{table}

\begin{table}
\caption{The eigenvalues on the unit square,
$n=8+4|x|$.}
\begin{center} \footnotesize
\begin{tabular}{ccccccc}\hline
$j$&$H$&$h$&$k_{j,H}$&$k_{j,h}$&$k_{j}^{h}$\\\hline
1&  $\frac{\sqrt2}{4}$& $\frac{\sqrt2}{16}$&    2.6153609845&   2.6036641359&  2.6037514197&\\
1&  $\frac{\sqrt2}{8}$& $\frac{\sqrt2}{32}$&    2.6044614497&   2.6036117648&   2.6036122672&\\
1&  $\frac{\sqrt2}{16}$& $\frac{\sqrt2}{64}$&    2.6036641359&   2.6036084278&   2.6036084300&\\
1&  $\frac{\sqrt2}{16}$&    $\frac{\sqrt2}{128}$&   2.6036641359&   2.6036082152&   2.6036082174&\\
1&  $\frac{\sqrt2}{32}$&    $\frac{\sqrt2}{256}$&   2.6036117648&   2.6036082015&   2.6036082016&\\
2,3&  $\frac{\sqrt2}{4}$& $\frac{\sqrt2}{16}$&    3.3100272777&   3.2863070722&   3.2876330633&\\
2,3&  $\frac{\sqrt2}{8}$& $\frac{\sqrt2}{32}$&    3.2888679496&   3.2861317830&   3.2861379266&\\
2,3&  $\frac{\sqrt2}{16}$&    $\frac{\sqrt2}{64}$&   3.2863070722&   3.2861205736&   3.2861205986&\\
2,3&  $\frac{\sqrt2}{16}$&    $\frac{\sqrt2}{128}$&   3.2863070722&   3.2861198684&   3.2861198936&\\
2,3&  $\frac{\sqrt2}{32}$&    $\frac{\sqrt2}{256}$&   3.2861317830&   3.2861198243&   3.2861198235\\
4&  $\frac{\sqrt2}{4}$& $\frac{\sqrt2}{16}$&    3.8113648152&   3.7967178330&   3.7980244032&\\
4&  $\frac{\sqrt2}{8}$& $\frac{\sqrt2}{32}$&    3.7987461560&   3.7965708629&   3.7965804312&\\
4&  $\frac{\sqrt2}{16}$&    $\frac{\sqrt2}{64}$&   3.7967178330&   3.7965612849&    3.7965613281&\\
4&  $\frac{\sqrt2}{16}$&    $\frac{\sqrt2}{128}$&   3.7967178330&   3.7965606792&   3.7965607226&\\
4&  $\frac{\sqrt2}{32}$&    $\frac{\sqrt2}{256}$&   3.7965708629&   3.7965606412&   3.7965606413&\\
\\\hline
\end{tabular}
\end{center}
\end{table}

 \begin{figure}
\centering
\includegraphics[width=0.6\textwidth]{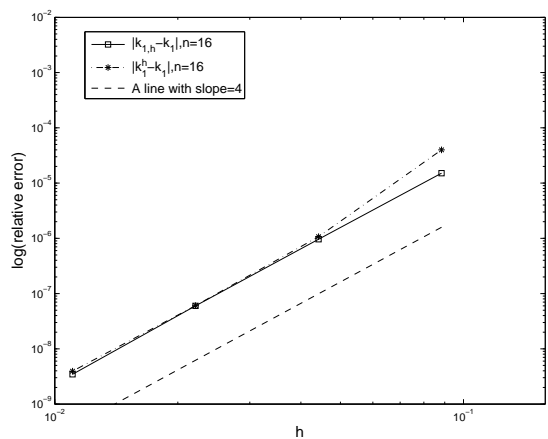}
\caption{{ Relative error curves on the unit square with $n=16$}}
 \end{figure}
\begin{figure}
\centering
\includegraphics[width=0.6\textwidth]{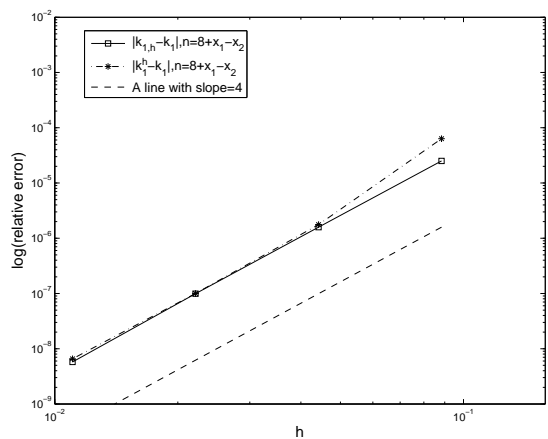}
\caption{{ Relative error curves on the unit square with $n=8+x_1-x_2$}}
 \end{figure}
 \begin{figure}
\centering
\includegraphics[width=0.6\textwidth]{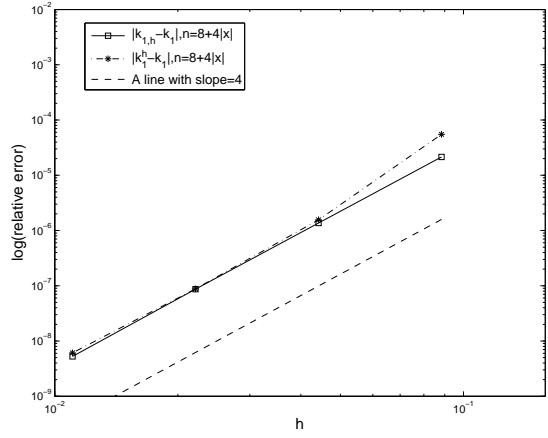}
\caption{{ Relative error curves on the unit square with $n=8+4|x|$}}
 \end{figure}

\begin{table}
\caption{The eigenvalues on the L-shaped
domain, $n=16$.}
\begin{center} \footnotesize
\begin{tabular}{ccccccc}\hline
$j$&$H$&$h$&$k_{j,H}$&$k_{j,h}$&$k_{j}^{h}$\\\hline
1&  $\frac{\sqrt2}{4}$& $\frac{\sqrt2}{16}$&    1.4968609397&   1.4802425308&   1.4806917664\\
1&  $\frac{\sqrt2}{8}$& $\frac{\sqrt2}{32}$&    1.4850588790&   1.4780405044&   1.4781253577\\
1&  $\frac{\sqrt2}{16}$&    $\frac{\sqrt2}{64}$&    1.4802425308&   1.4770116526&   1.4770299626\\
1&  $\frac{\sqrt2}{32}$&    $\frac{\sqrt2}{128}$&   1.4780405044&   1.4765287608&   1.4765327848\\
1&  $\frac{\sqrt2}{32}$&    $\frac{\sqrt2}{256}$&   1.4780405044&   ---&            1.4763071753\\
2&  $\frac{\sqrt2}{4}$& $\frac{\sqrt2}{16}$&    1.5755693911&   1.5698997470&   1.5699245378\\
2&  $\frac{\sqrt2}{8}$& $\frac{\sqrt2}{32}$&    1.5705428361&   1.5697715392&   1.5697719333\\
2&  $\frac{\sqrt2}{16}$&    $\frac{\sqrt2}{64}$&    1.5698997470&   1.5697385282&   1.5697385527\\
2&  $\frac{\sqrt2}{32}$&    $\frac{\sqrt2}{128}$&   1.5697715392&   1.5697293608&   1.5697293627\\
2&  $\frac{\sqrt2}{32}$&    $\frac{\sqrt2}{256}$&   1.5697715392&            ---&   1.5697267819\\
3&  $\frac{\sqrt2}{4}$& $\frac{\sqrt2}{16}$&    1.7158442839&   1.7061971095&   1.7057822065\\
3&  $\frac{\sqrt2}{8}$& $\frac{\sqrt2}{32}$&    1.7077892476&   1.7055794987&   1.7055176311\\
3&  $\frac{\sqrt2}{16}$&    $\frac{\sqrt2}{64}$&    1.7061971095&   1.7052949847&   1.7052821504\\
3&  $\frac{\sqrt2}{32}$&    $\frac{\sqrt2}{128}$&   1.7055794987&   1.7051612774&   1.7051584514\\
3&  $\frac{\sqrt2}{32}$&    $\frac{\sqrt2}{256}$&   1.7055794987&            ---&   1.7050946832\\
4&  $\frac{\sqrt2}{4}$& $\frac{\sqrt2}{16}$&    1.7871637125&   1.7831474550&   1.7831764131\\
4&  $\frac{\sqrt2}{8}$& $\frac{\sqrt2}{32}$&    1.7834444240&   1.7831207586&   1.7831207597\\
4&  $\frac{\sqrt2}{16}$&    $\frac{\sqrt2}{64}$&    1.7831474550&   1.7831171038&   1.7831170947\\
4&  $\frac{\sqrt2}{32}$&    $\frac{\sqrt2}{128}$&   1.7831207586&   1.7831163288&   1.7831163281\\
4&  $\frac{\sqrt2}{32}$&    $\frac{\sqrt2}{256}$&   1.7831207586&            ---&   1.7831161310\\
\\\hline
\end{tabular}
\end{center}
\end{table}

\begin{table}
\caption{The eigenvalues on the L-shaped
domain, $n=8+x_1-x_2$.}
\begin{center} \footnotesize
\begin{tabular}{ccccccc}\hline
$j$&$H$&$h$&$k_{j,H}$&$k_{j,h}$&$k_{j}^{h}$\\\hline
1&  $\frac{\sqrt2}{4}$& $\frac{\sqrt2}{16}$&    2.3273092209&   2.3069609372&   2.3077590615\\
1&  $\frac{\sqrt2}{8}$& $\frac{\sqrt2}{32}$&    2.3126825597&   2.3043812480&   2.3045405243\\
1&  $\frac{\sqrt2}{16}$&    $\frac{\sqrt2}{64}$&    2.3069609372&   2.3032153953&   2.3032156195\\
1&  $\frac{\sqrt2}{32}$&    $\frac{\sqrt2}{128}$&  2.3043812480&   2.3026187675&    2.3026263111\\
1&  $\frac{\sqrt2}{32}$&    $\frac{\sqrt2}{256}$&  2.3043812480&   ---&             2.3023645323\\
2&  $\frac{\sqrt2}{4}$& $\frac{\sqrt2}{16}$&    2.4080210782&   2.3960501482&   2.3961193118\\
2&  $\frac{\sqrt2}{8}$& $\frac{\sqrt2}{32}$&    2.3973860151&   2.3957859638&   2.3957867865\\
2&  $\frac{\sqrt2}{16}$&    $\frac{\sqrt2}{64}$&    2.3960501482&   2.3957181881&   2.3957182403\\
2&  $\frac{\sqrt2}{32}$&    $\frac{\sqrt2}{128}$&   2.3957859638&   2.3956993970&   2.3956994013\\
2&  $\frac{\sqrt2}{32}$&    $\frac{\sqrt2}{256}$&   2.3957859638&            ---&    2.3956940981\\
5,6&    $\frac{\sqrt2}{4}$& $\frac{\sqrt2}{16}$&    2.9097574473&    2.9272578941&   2.9226353850\\
&&&                                            $\pm$0.6019761082$i$&   $\pm$0.5686107059$i$&   $\pm$0.5684665341$i$\\
5,6&    $\frac{\sqrt2}{8}$& $\frac{\sqrt2}{32}$&    2.9290487145 &      2.9257100894   &  2.9252369019 \\
&&&                                               $\pm$0.5742896976$i$&   $\pm$0.5664333903$i$&   $\pm$0.5664645461$i$\\
5,6&    $\frac{\sqrt2}{16}$&    $\frac{\sqrt2}{64}$&    2.9272578941 &   2.9249315763&  2.9248341507\\
&&&                                               $\pm$0.5686107059$i$&    $\pm$0.5654493727$i$&  $\pm$0.5654547602$i$\\
5,6&    $\frac{\sqrt2}{32}$&    $\frac{\sqrt2}{128}$&   2.9257100894 &    2.9245612175&       2.9245399519\\
&&&                                                 $\pm$0.5664333903$i$&   $\pm$0.5649922109$i$& $\pm$0.5649932074$i$\\
5,6&    $\frac{\sqrt2}{32}$&    $\frac{\sqrt2}{256}$&   2.9257100894 &                       ---&       2.9243585119\\
&&&                                                 $\pm$0.5664333903$i$&                    & $\pm$0.5647798350$i$\\
\\\hline
\end{tabular}
\end{center}
\end{table}

\begin{table}
\caption{The eigenvalues on the L-shaped
domain, $n=8+4|x|$.}
\begin{center} \footnotesize
\begin{tabular}{ccccccc}\hline
$j$&$H$&$h$&$k_{j,H}$&$k_{j,h}$&$k_{j}^{h}$\\\hline
1&  $\frac{\sqrt2}{4}$& $\frac{\sqrt2}{16}$&    1.8901335914&   1.8737882252&   1.8742672861\\
1&  $\frac{\sqrt2}{8}$& $\frac{\sqrt2}{32}$&    1.8784255776&   1.8716730328&   1.8717681399\\
1&  $\frac{\sqrt2}{16}$&    $\frac{\sqrt2}{64}$&    1.8737882252&   1.8706818414&   1.8707028764\\
1&  $\frac{\sqrt2}{32}$&    $\frac{\sqrt2}{128}$&   1.8716730328&   1.8702153831&   1.8702200612\\
1&  $\frac{\sqrt2}{32}$&    $\frac{\sqrt2}{258}$&   1.8716730328&           ---&    1.8700019875\\
2&  $\frac{\sqrt2}{4}$& $\frac{\sqrt2}{16}$&    2.0064347086&   1.9980861736&   1.9981277663\\
2&  $\frac{\sqrt2}{8}$& $\frac{\sqrt2}{32}$&    1.9990284890&   1.9978995077&   1.9978999562\\
2&  $\frac{\sqrt2}{16}$&    $\frac{\sqrt2}{64}$&    1.9980861736&   1.9978514092&   1.9978514411\\
2&  $\frac{\sqrt2}{32}$&    $\frac{\sqrt2}{128}$&   1.9978995077&   1.9978380258&   1.9978380285\\
2&  $\frac{\sqrt2}{32}$&    $\frac{\sqrt2}{256}$&   1.9978995077&            ---&   1.9978342423\\
3&  $\frac{\sqrt2}{4}$& $\frac{\sqrt2}{16}$&    2.1999897619&   2.1946361010&   2.1946838299\\
3&  $\frac{\sqrt2}{8}$& $\frac{\sqrt2}{32}$&    2.1950294577&   2.1946045829&   2.1946047250\\
3&  $\frac{\sqrt2}{16}$&    $\frac{\sqrt2}{64}$&   2.1946361010&   2.1946010623&   2.1946010528\\
3&  $\frac{\sqrt2}{32}$&    $\frac{\sqrt2}{128}$&   2.1946045829&   2.1946004195&   2.1946004187\\
3&  $\frac{\sqrt2}{32}$&    $\frac{\sqrt2}{256}$&   2.1946045829&            ---&   2.1946002595\\
4&  $\frac{\sqrt2}{4}$& $\frac{\sqrt2}{16}$&    2.2917600981&   2.2632424103&   2.2626212924\\
4&  $\frac{\sqrt2}{8}$& $\frac{\sqrt2}{32}$&    2.2677627165&   2.2614697394&   2.2615130804\\
4&  $\frac{\sqrt2}{16}$&    $\frac{\sqrt2}{64}$&    2.2632424103&   2.2606440982&   2.2606589214\\
4&  $\frac{\sqrt2}{32}$&    $\frac{\sqrt2}{128}$&   2.2614697394&   2.2602536398&   2.2602568660\\
4&  $\frac{\sqrt2}{32}$&    $\frac{\sqrt2}{256}$&   2.2614697394&            ---&   2.2600736729\\
\\\hline
\end{tabular}
\end{center}
\end{table}

\begin{figure}
\centering
\includegraphics[width=0.6\textwidth]{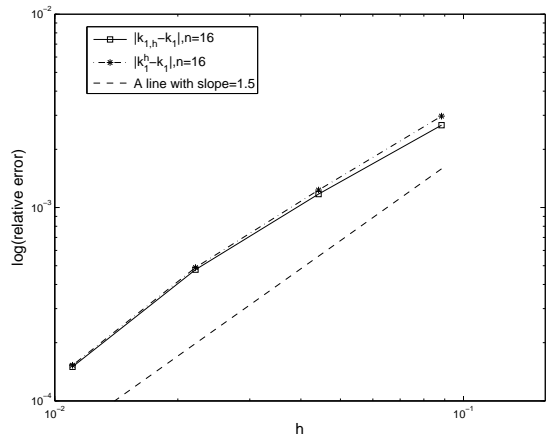}
\caption{{ Relative error curves on the L-shaped domain with $n=16$}}
 \end{figure}
\begin{figure}
\centering
\includegraphics[width=0.6\textwidth]{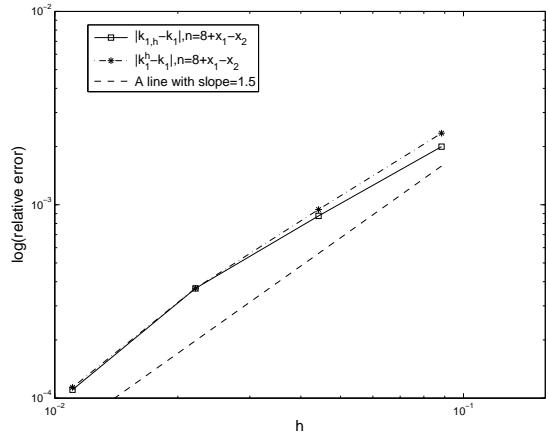}
\caption{{ Relative error curves on the L-shaped domain with $n=8+x_1-x_2$}}
 \end{figure}
 \begin{figure}
\centering
\includegraphics[width=0.6\textwidth]{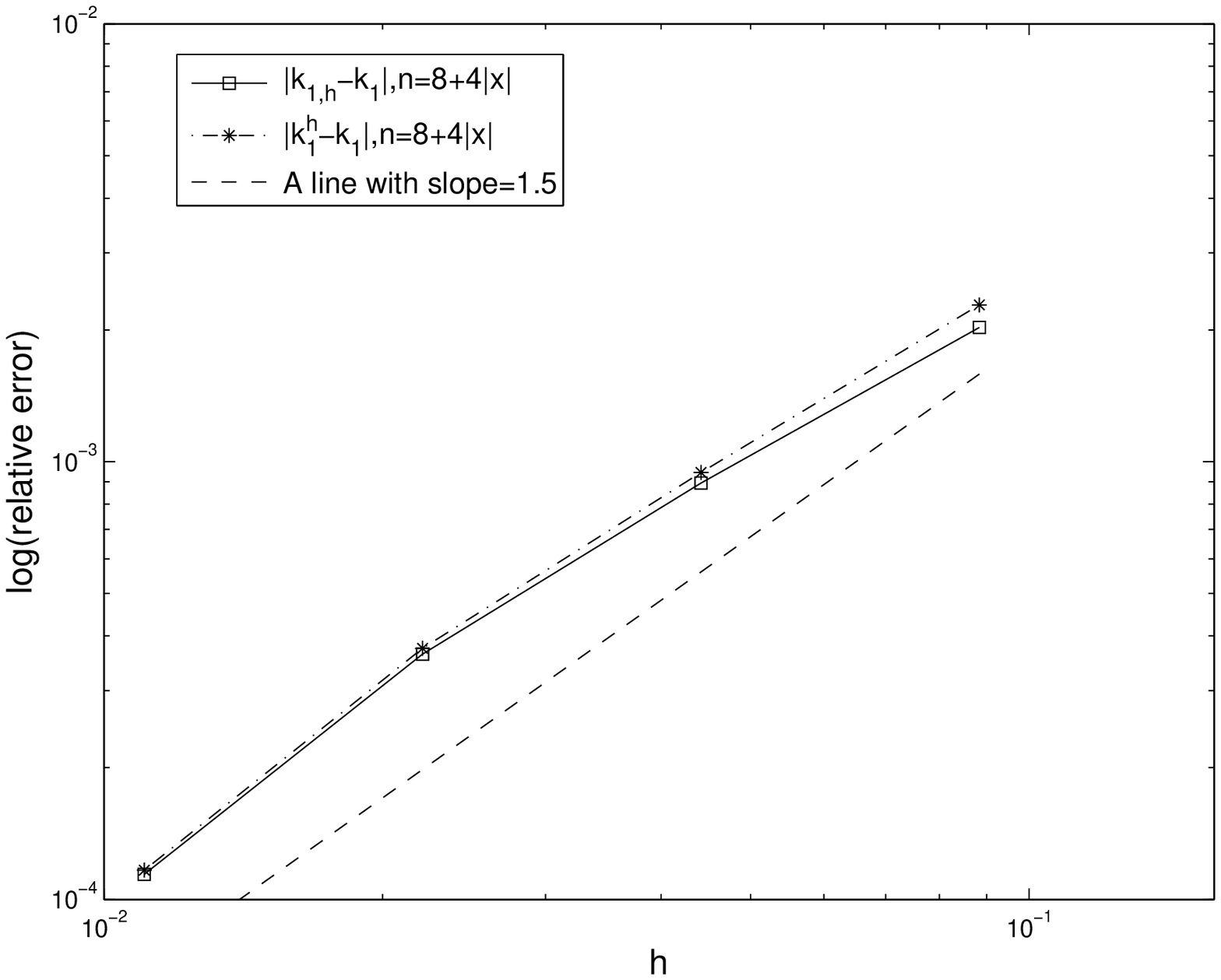}
\caption{{ Relative error curves on the L-shaped domain with $n=8+4|x|$}}
 \end{figure}

 \indent From Tables 1-2 we see that our results are the same as the numerical results in \cite{13}. Tables 1-3 tell us both traditional computational method and two-grid discretization have fast rates of convergence, and achieve the same convergence order in computing the eigenvalues. We also find that for real eigenvalues, as $h$ decreases gradually, $k_{j,h}$ and $k_j^h$ decrease simultaneously to approximate to the exact solution. So we have reason to believe that the two-grid discretization is efficient. It's worthwhile to note that $r=1$ and $(u,w)\in H_{0}^4\times H_{0}^{3}(\Omega)$ on unite square according to the regularity theory. When the ascent $\alpha=1$, based on (\ref{s3.5}) and (\ref{s3.9}), convergence order of eigenvalue approximation $k_{j,h}$ is 4; according to $(\ref{s4.5})$ and $(\ref{s4.6})$, while $H^3\lesssim h^2$, we also can make the convergence order of $k_{j}^h$ achieving 4. Using $k_{1,\frac{\sqrt{2}}{256}}$ as the exact value $k_1$, we plot relative error curves. As expected, Figures 1-3 tell us the convergence order satisfies the conclusion we get above.\\

 \indent{\bf Example 2.}~~We consider the L-shaped domain $\Omega=(-1,1)\times(-1,1)\backslash ([0,1]\times[-1,0])$
 with the function $n(x)=16$, $n(x)=8+x_1-x_2$ and $n(x)=8+4|x|$.
 The numerical results are presented in Tables 4-6.
  We find that when $h=\frac{\sqrt{2}}{256}$, the finite element discretization method is out of memory during computing,
  but the two-grid discretization method can finish the computational mission successfully. We infer it is due to the eigs
  command which may occupy much more physical memory. Therefore we believe the two-grid discreization is more advantage than the traditional one.
  Taking $k_{1}^{\frac{\sqrt{2}}{256}}$ as the exact value $k_1$, we plot relative error curves. Figures 4-6 suggest that the convergence rates
  are slow relatively since the smoothness of functions $(u,w)$ is weak for the L-shaped domain.\\
%\section*{Conflict of Interests}
%The authors declare that there is no conflict of interests regarding the publication of this paper.
\section*{Acknowledgments}
This work is supported by the National Science Foundation of China
(Grant No. 11561014 )


\begin{thebibliography}{s10}
\bibitem{1}{\sc F. Cakoni, M. Cayoren, D. Colton}, {\em Transmission eigenvalues
and the nondestructive testing of dielectrics}, Inverse Problems,
24, 065016 (2008)

\bibitem{2}{\sc F. Cakoni, D. Gintides, H. Haddar}, {\em The existence of an
infinite discrete set of transmission eigenvalues}, SIAM J. Math.
Anal., 42(2010), pp. 237-255

\bibitem{3}{\sc Bergfield. Justin P., Joshua D. Barr, Charles A. Stafford}, {\em Transmission eigenvalue distributions in highly conductive molecular junctions},
Beilstein journal of nanotechnology 3(1) (2012), pp. 40-51.

\bibitem{colton2}D. Colton, P. Monk, J. Sun, \textit{ Analytical and computational
methods for transmission eigenvalues.} Inverse Problems, \textit{26}
(2010) 045011.

\bibitem{ji}X. Ji, J. Sun, T. Turner, \textit{  Al- gorithm 922: a mixed finite
element method for Helmholtz transmission eigenvalues.} ACM
Transaction on Math. Soft., \textit{38} (2012) 29:1--8.

\bibitem{monk}P. Monk, J. Sun, \textit{ Finite element methods of Maxwell
transmission eigenvalues.} SIAM J. Sci. Comput., \textit{34} (2012)
B247--264.





\bibitem{8}{\sc F. Cakon, P. Monk, J. Sun}, {\em Error analysis for the finite element approximation of transmission eigenvalues},
Comput. Methods Appl. Math., 14 (2014), pp. 419-427.

\bibitem{9}{\sc X. Ji, J. Sun, H. Xie}, {\em A multigrid method for Helmholtz transmission eigenvalue
problems}, J. Sci. Comput., 60(2014), pp. 276-294.

\bibitem{sun1}{\sc J. Sun}, {\em Iterative methods for transmission eigenvalues}, SIAM J.
Numer. Anal., 49(2011), pp. 1860-1874.

\bibitem{10}{\sc J. Xu}, {\em A new class of iterative methods for nonselfadjoint or indefinite
problems}, SIAM J. Numer. Anal., 29 (1992) pp. 303-319.

\bibitem{11}{\sc X. Dai, A. Zhou}, {\em Three-scale finite element discretizations
for quantum eigenvalue problems}, SIAM J. Numer. Anal., 46(1)
(2008), pp. 295-324.

\bibitem{12}{\sc J. Xu and A. Zhou}, {\em A two-grid discretization scheme for eigenvalue problems}, Math. Comput.,70 (2001), pp. 17-25.

\bibitem{19}{\sc K. Kolman}, {\em A two-level method for nonsymmetric eigenvalue problems},
Acta Math. Appl. Sin. Engl. Ser., 21(1)(2005), pp. 1-12.


\bibitem{zhou}{\sc J. Zhou, X. Hu, L. Zhong, S. Shu, L. Chen},
{\em Two-grid methods for Maxwell eigenvalue problems}, SIAM J.
Numer. Anal., 52(2014), pp.2027-2047.

\bibitem{yang2}{\sc Y. Yang, H. Bi, J. Han, Y. Yu }, {\em The shifted-inverse iteration based on the multigrid
discretizations
 for
eigenvalue problems}, SIAM J. Sci. Comput., 37(6) (2015), pp.
A2583-A2606.

\bibitem{yang3}{\sc Y. Yang, J. Han}, {\em The multilevel mixed finite element discretizations
based on local defect-correction for the Stokes eigenvalue problem},
Comput. Methods Appl. Mech. Engrg., 289 (2015), pp. 249-266.

\bibitem{13}{\sc Y. Yang, H. Bi, J. Han}, {\em Error estimates and
a two grid scheme for approximating transmission eigenvalues},
arXiv:1506.06486 V2 [math. NA] 2 Mar 2016.

\bibitem{14}{\sc I. Babuska, J.E. Osborn}, {\em Eigenvalue Problems}, in: P.G. Ciarlet, J.L. Lions,(Ed.),
Finite Element Methods (Part 1), Handbook of Numerical Analysis,
vol.2, Elsevier Science Publishers, North-Holand, 1991, pp. 640-787.

\bibitem{15}{\sc S.C. Brenner, L.R. Scott}, {\em The Mathematical Theory of Finite Element Methods}, 2nd ed.,
Springer-Verlag, New york, 2002.

\bibitem{16}{\sc P.G. Ciarlet}, {\em Basic error estimates for elliptic proplems}, in: P.G. Ciarlet, J.L. Lions, (Ed.),
Finite Element Methods (Part1), Handbook of Numerical Analysis,
vol.2, Elsevier Science Publishers, North-Holand, 1991, pp.21-343.

\bibitem{cakoni3}{\sc F. Cakoni, H. Haddar}, {\em On the existence of
transmission eigenvalues in an inhomogeneous medium}, Appl. Anal.
88(4)(2009), pp. 475-493.


\bibitem{7}{\sc B.P. Rynne, B.D. Sleeman}, {\em The interior transmission problem
and inverse scattering from inhomogeneous media}, SIAM J. Math.
Anal., 22 (1991), pp. 1755-1762 (1992).



\bibitem{17}{\sc H. Blum, R. Rannacher}, {\em On the boundary value problem of the biharmonic operator on domains with
angular corners}, Math. Method Appl. Sci., 2(1980), pp. 556-581

\bibitem{18}{\sc P. Grisvard}, {\em Elliptic problems in nonsmooth domains},
Pitman, London, 1985.

\bibitem{chen}{\sc L. Chen}, {\em An integrated finite element method package in MATLAB, Technical
Report},
University of California at Irvine, California, 2009.







\end{thebibliography}
\end{document}